\let\nofiles\relax

\documentclass[twoside]{article}

\usepackage{subfigure}
\usepackage{mathrsfs}
\usepackage[sumlimits]{amsmath}
\usepackage{amsthm,amssymb,amsfonts}
\usepackage{graphicx,color}
\usepackage{algorithm}
\usepackage{algpseudocode}
\usepackage{titlesec}
\usepackage{float}
\usepackage{caption}
\usepackage{extarrows}

\usepackage[
linktocpage=true,
colorlinks=true,
linkcolor=blue,
citecolor=blue,
urlcolor=blue
]{hyperref}
\usepackage[numbers,sort&compress]{natbib}

\allowdisplaybreaks

\catcode`@=11
\long\def\@makefntext#1{\noindent #1}
\newskip\tabcentering\tabcentering=1000pt plus 1000pt minus 1000pt

\def\MCH#1#2{\setbox0=\hbox{\raise#1\hbox{#2}}\smash{\box0}}% move char

\def\evenhead{}
\def\oddhead{}
\def\@evenhead{\hbox to\textwidth{\footnotesize\rm\thepage\hfill{\small\it\evenhead}}}
\def\@oddhead{\hbox to\textwidth{\footnotesize{\small\it\oddhead}\hfill\footnotesize\rm\thepage}}
\def\@evenfoot{}
\def\@oddfoot{}

\catcode`@=12
\font\bb=cmbx8%
\font\small=cmr8

\font\st=cmbx10 scaled\magstep2%
\def\title#1{{\noindent\st#1}}%
\def\author#1{\vspace*{0.3 true cm}{\noindent\bf #1}}%
\def\abstract#1{\vspace*{1 true cm}{\noindent{\footnotesize{\bb Abstract}\hspace{4true mm}#1}}}%
\def\keywords#1{\vspace{5mm}\noindent{\footnotesize{\bb Keywords}\hspace{4true mm}#1}}%

\newtheorem{definition}{Definition}[section]

\newtheorem{lemma}[definition]{Lemma}
\newtheorem{theorem}[definition]{Theorem}

\newtheorem*{remark*}{Remark}

\newtheoremstyle{noparens}%
{}{}%
{\itshape}{}%
{\bfseries}{\bf .}%
{ }%
{\thmname{#1}\thmnumber{ #2}\mdseries\thmnote{#3}}

\theoremstyle{noparens}

\theoremstyle{definition}
\newtheorem{remark}[definition]{Remark}

\numberwithin{equation}{section}
\numberwithin{figure}{section}

\titleformat*{\section}{\large \bfseries}
\titleformat*{\subsection}{\bfseries}
\titleformat*{\subsubsection}{\normalsize \bfseries}
\newcommand{\citeup}[1]{\textsuperscript{\cite{#1}}}

\begin{document}
%%%%for editor%%%%%%%%%%
%\topbottom{1}{?}{40}{1}{2024}\\
%{\footnotesize https://doi.org/10.1007/s10255-022-1020-9\\\vspace{-5mm}\\
%	http://www.ApplMath.com.cn \& www.SpringerLink.com}
%
%\vskip-1.35cm
%
%
%\begin{figure}[h] \hskip9.9cm {
%		\includegraphics[height=1.4cm,width=4.5cm]{nactmark2024.eps}}
%\end{figure}
%
%\vskip35pt
%%%%for editor%%%%%%%%%%

%%%%for editor%%%%%%%%%%

%\def\oddhead{Polynomial Stability of an Elastic Plate}%odd page header��title
%\def\evenhead{Y.N. SUN, Q. ZHANG}%even page header ��abbreviation of all the author's name, e.g., Z.W. Ning

\title{Polynomial Stability of an Elastic Thin Plate on Non-Smooth Domain}%

\author{SUN Ya-nan,ZHANG Qiong}

%\institute{$^1$School of Mathematics and Statistics, Beijing Key Laboratory on MCAACI, Beijing Institute of Technology, Beijing, 100081, China  ($^\dag$E-mail: zhangqiong@bit.edu.cn)}

\vskip10pt

%\no{\footnotesize{\bf This paper is dedicated to Professor Philippe
%		G. Ciarlet on the occasion of his 80th birthday.}}

\vskip-10pt

\abstract{This paper studies the polynomial stabilization of an elastic plate with dynamical boundary conditions on a non-smooth domain. To deal with  the possible loss of solution regularity induced by boundary singularities, we formulate the problem as a precise variational framework. We prove that for domains with sufficiently small corner angles, the system retains the polynomial decay rate under standard geometric control conditions. In cases where larger corner angles lead to a significant regularity loss, we show that polynomial stability is recovered by introducing a  feedback control at the corners.}


%\mr{ }%2000 MR Subject Classification

\section{Introduction}\label{Introduction}

Let  $\Omega$ be  a bounded open connected convex domain in $\mathbb{R}^2$.
Its boundary $\partial\Omega$ is assumed to be a curvilinear polygon of class
$C^{1,1}$  with a finite set of corner points   $P = \{P_i : i = 1, 2, \dots, p\}$. The boundary
is composed of two relatively open subsets $\Gamma_0$ and $\Gamma_1$, where
$\Gamma_0$ is non-empty and has positive boundary measure. The  two subsets $\Gamma_0$ and $\Gamma_1$ are either disjoint or share two common endpoints, both of which are corner points.

We consider a linear hybrid model of an elastic thin plate on $\Omega$.
The plate is clamped along the portion $\Gamma_0$  of the boundary. The complementary part $\Gamma_1$ is free  but is bordered by a flange possessing both mass and rotational inertia.  The vibration  of the plate is governed by the following system \cite{JE}:
\begin{equation}\label{eq1}
	\begin{cases}
		u_{tt}+\Delta^2 u = 0, \qquad&\mbox{in}\;\Omega,\\
		J\partial_{\nu}u_{tt} + \mathcal{B}_1 u = -d_1 \partial_{\nu}u_t,&\text{on}\; \Gamma_1 ,\\
		\rho u_{tt} - \mathcal{B}_2 u = -d_2 u_t, &\text{on}\; \Gamma_1 ,\\
		u=\partial_{\nu} u = 0, &\text{on}\; \Gamma_0,\\
		u(0,x) = u_0(x),\;\;u_t(0,x) = u_1(x), \qquad&\mbox{in}\; \Omega,
	\end{cases}
\end{equation}
%with the boundary condition on $\Gamma_0$,
%\begin{equation}\label{clamped}
%u=\partial_{\nu} u = 0,\qquad  \text{on}\; \Gamma_0.
%\end{equation}
%In the case of a plate which is simply supported on $\Gamma_0$, \eqref{clamped} is replaced by (\cite{Lions})
%\begin{equation}\label{simply}
%u=\mathcal B_1 u =0,\qquad\mbox{on}\;\Gamma_0.
%\end{equation}
where $\mathcal{B}_{1}$ and  $\mathcal{B}_{2}$ are boundary operators associated with the plate equation:
\begin{align*}\displaystyle
	&\mathcal{B}_{1}u=\Delta u+(1-\mu)\Big[2\nu_{1}\nu_{2}\frac{\partial^{2}u}{\partial x_{1}\partial x_{2}}-\nu_{1}^{2}\frac{\partial^{2}u}{\partial x_{2}^{2}}-\nu_{2}^{2}\frac{\partial^{2}u}{\partial x_{1}^{2}}\Big],\\ \noalign{\medskip}  \displaystyle
	&\mathcal{B}_{2}u=\partial_{\nu}\Delta u+(1-\mu)\partial_{\tau}\Big[(\nu_{1}^{2}-\nu_{2}^{2})\frac{\partial^{2}u}{\partial x_{1}\partial x_{2}}+ \nu_{1}\nu_{2}\big(\frac{\partial^{2}u}{\partial x_{2}^{2}}-\frac{\partial^{2}u}{\partial x_{1}^{2}}\big)\Big],
\end{align*}
$\nu=(\nu_{1},\;\nu_{2})$ and $\tau=(-\nu_{2},\;\nu_{1})$ denote the unit outer normal vector and the unit tangent vector on the boundary, respectively. $ 0< \mu <\frac{1}{2}$ is the Poisson ratio of elasticity, $\rho>0$  is the linear boundary
density, and $J > 0 $ is the bending moment of inertia per unit length of the boundary.
The initial data are given   by $ u_0,\; u_1$, and the control parameters $d_1$ and $d_2$  are assumed to be positive.

The energy of the system \eqref{eq1}  is defined by
\begin{align}\label{energy}
	E(t)=\frac{1}{2}\Big[a(u(t))+\|u_{t}\|^2_{L^2(\Omega)}
	+J\|\partial_{\nu}u_{t}\|^{2}_{L^2(\Gamma_{1})} +\rho\|u_{t}\|^{2}_{L^2(\Gamma_{1})} \Big],\ \ t\geq 0,
\end{align}
where the bilinear form $a(u):=a(u,u)$ is
\begin{align}\label{bilinear-a}
	\begin{split}
		a(u,v)=&\int_{\Omega}
		\bigg[\frac{\partial^{2}u}{\partial x_{1}^{2}}\frac{\partial^{2}v}{\partial x_{1}^{2}}+\frac{\partial^{2}u}{\partial x_{2}^{2}}\frac{\partial^{2}v}{\partial x_{2}^{2}}
		+\mu\bigg(\frac{\partial^{2}u}{\partial x_{1}^{2}}\frac{\partial^{2}v}{\partial x_{2}^{2}}+\frac{\partial^{2}u}{\partial x_{2}^{2}}\frac{\partial^{2}v}{\partial x_{1}^{2}}\bigg) \\
		& \qquad+2(1-\mu)\frac{\partial^{2}u}{\partial x_{1}\partial x_{2}}\frac{\partial^{2}v}{\partial x_{1}\partial x_{2}}
		\bigg]dx_{1}dx_{2}, \;\; \forall\; u, v\in H^2(\Omega).
	\end{split}
\end{align}
\noindent
%A direct computation yields
%$$
%\frac{d}{dt}E(t)
%=-d_1\|u_{t}\|_{\Gamma_{1}}^{2}
%-d_2\|\partial_{\nu}u_{t}\|_{\Gamma_{1}}^{2}- \sum_{P_i \in \partial \Omega \setminus  {\Gamma_0}} |u_t(P_i)|^2,$$
%which implies that system \eqref{eq1}  is dissipative.

There exists an extensive literature on the stabilization of elastic plates (see, e.g., \cite{Robbiano, Ammari1, JE, Littman, Liu-Liu1, LiuYan, Rao1, Rao2, Tucsnak2, Tebou} and references therein). It is known \citeup{LiuYan, Rao2} that for the elastic plate with dynamical boundary conditions (i.e., $J,\rho > 0$ in \eqref{eq1}), the energy is not exponentially stable. Specifically, Rao proved that under sufficient smoothness of the boundary $\partial\Omega$ and certain geometric conditions on $\Omega$ (typically the Geometric Control Condition \citeup{Lebeau}), the semigroup associated with system \eqref{eq1} is polynomially stable with decay rate $t^{-1/2}$.

A natural subsequent question is: what happens when the domain is not smooth? For instance, if $\Omega$ is a  polygon or  curvilinear polygon. Studies on the stabilization of waves, plates, shells, or coupled systems on non-smooth domains can be found in Ref. \cite{ammari, chen1, chen2}, among others. The stabilization problem on non-smooth domains presents two main difficulties. First, the solution may lose regularity for certain domain geometries \citeup{blum, dauge, grisvard}, necessitating the analysis of the system in a weaker space and a reformulation of its functional framework. Second, due to this loss of regularity, a new integration-by-parts formula must be established, which introduces additional terms arising from the corners. Consequently, feedback control acting on the corners must also be incorporated.

In this paper, we analyze the long-time behavior of system \eqref{eq1}. To deal with the potential loss of solution regularity, we formulate a variational framework. We note that the regularity of the solution depends on the angles at the corners \citeup{blum}. When these angles are sufficiently small, we obtain polynomial stability with the same decay rate as for smooth domains. In cases where regularity of the solution is lost, we show that introducing feedback at the corners restores polynomial stability. This result complements analysis of the long-time dynamics of elastic plates and will   facilitate future study of the elastic system on more general domains.

The remainder of this paper is organized as follows. Section 2 introduces the preliminaries and main result.   Section 3 is dedicated to proving the polynomial stability.

In the subsequent analysis, the symbol $C$ denotes a generic positive constant whose value may differ from line to line or even within the same line.
%We employ the notation $f\lesssim g$  to mean $f\leq Cg$, and  $f\sim g$ to indicate $C_1 g\leq f\leq C_2 g$ for some positive constants $C_1$ and $C_2$.
For brevity, the $L^2(\Omega)$-norm and $L^2(\Gamma_1)$-norm will be denoted by $\|\cdot\|_{\Omega}$ and $\|\cdot\|_{\Gamma_1}$, respectively.

\section{Preliminaries and main results}

This section presents the necessary  preliminaries and main results. We first define the space
\begin{align*}
	V=\big\{u\in H^{2}(\Omega)\; \big|\; \ u|_{\Gamma_{0}}= \partial_{\nu} u|_{\Gamma_{0}}=0\big\}.
\end{align*}

Since the boundary of $\Omega$ is a curvilinear
polygon,   establishing the required regularity results and the corresponding integration-by-parts formula for the bi-laplacian operator on this non-smooth domain is essential. It is  known that the solutions to the
fourth-order boundary problems on the nonsmooth domain can be divided
into a regular part and a singular part (see, e.g.,
\cite{blum,dauge,grisvard,nicaise92}).
The singularities near a ``curved'' corner are determined  by the boundary conditions, Poisson ratio $\mu$ and the interior angle between the tangents to the two boundary segments meeting at that corner.

Accordingly, we  consider two cases: (i) when the interior angle at each critical boundary point is sufficiently small, and (ii)  the complementary case. The corresponding regularity results and integration-by-parts formulae are presented in the following two lemmas.

\medskip

\begin{lemma} \label{Le2.2} (See Ref. \cite{blum, dauge}) Suppose   the domain $\Omega$ satisfies  the following  condition.
	
	{\bf   (G)  } Let $\omega_j$ ($j=1,...,p$) denote the interior
	angles at the corners of $\Omega$ between two consecutive curve. Assume there exists an  $\omega_0$ such that $\omega_j<\omega_0$ for all $j=1,..,p$, where
	the minimal angle $\omega_0$ depends on Poisson ratio $\mu$ as specified in Theorem 2 of \cite{blum} (for example,
	$\omega_0\simeq52.054347...^\circ$ when $\mu = 0.3$).
	
	If  a function $u\in V$ satisfies the
	following   boundary value problem
	\begin{align}
		\label{am2} 
		\begin{cases}
			\Delta^2 u  \in L^2(\Omega),\\
			\mathcal B_1u,\;  \mathcal B_2 u \in L^2(\Gamma_1).
		\end{cases}
	\end{align}
	Then $u\in H^4(\Omega)$ and the following Green's  formula holds:
	\begin{align}
		\label{green} \int_{\Omega}\Delta^{2}u \:  v dx =  a(u,
		v ) +\int_{\partial\Omega}\mathcal B_{2}u\:  v d\Gamma -
		\int_{\partial\Omega}\mathcal B_{1}u\: \partial_{\nu_1} v  d\Gamma,  \;\;
		\forall\;  v \in H^{2}(\Omega).
	\end{align}
\end{lemma}

\medskip

\begin{remark}
	Lemma \ref{Le2.2}   describes  the regularity of the
	solution of the bi-laplacian operator under specific boundary conditions, which follows from Theorem 2 of \cite{blum} or Theorem 23.3 of \cite{dauge}.
	The regularity depends on both the boundary conditions and the interior angles at the corner points.
	Under Assumption (G), a direct computation yields
	$$
	\frac{d}{dt}E(t)
	=-d_1\|u_{t}\|_{\Gamma_{1}}^{2}
	-d_2\|\partial_{\nu}u_{t}\|_{\Gamma_{1}}^{2},$$
	which implies that system \eqref{eq1}  is dissipative.
\end{remark}

\medskip

In the complementary case, particularly when the interior angle at a corner violates Assumption (G), solutions to the bi-laplacian boundary value problem lose full regularity. To handle this, a weak frame is required. The fundamental distinction between fourth-order elliptic problems on smooth domains and those on domains with corners lies in the integration-by-parts formula.
The presence of corners introduces additional energy terms. Consequently, the standard integration-by-parts formula must be modified to account for these corner effects, which in turn necessitates supplementary feedback controls to ensure system dissipation.

\medskip

\begin{lemma} \label{lemma-green}
	Let $\partial\Omega$ be
	parameterized   counterclockwise. Then, for sufficiently
	smooth functions $u$ and $v$, we have
	\begin{align}
		\int_{\Omega}\Delta^{2}u v dx   =
		a(u,  v ) + \int_{\partial\Omega} (\mathcal B_{2}u v - \mathcal B_{1}u \partial_{\nu} v )d\Gamma\  + \,\displaystyle\sum\limits_{i=1}^{p}[M_{t}(u)(P_{i})]
		v (P_{i}),
		\label{eq4.1}
	\end{align}
	where the quantity
	$$
	[M_{t}(u)(P_{i})] = M (u)(P_{i}^{+}) - M (u)(P_{i}^{-})
	$$
	is the jump of $M_{t}(u)$ across  the corner $P_{i}$ in the direction of increasing arc length, and
	$$
	M(u) = (1-\mu)\bigg[(\nu_{1}^{2}-\nu_{2}^{2}){\frac{\partial^{2}u}{\partial x_{1}
			\partial x_{2}}}+\nu_{1}\nu_{2}\bigg({\frac{\partial^{2}u}{\partial x_{2}^{2}}}-
	{\frac{\partial^{2}u}{\partial x_{1}^{2}}}\bigg)\bigg]
	$$
	is the twisting moment.
\end{lemma}

\medskip

Comparing (\ref{green}) with (\ref{eq4.1}), the sum in (\ref{eq4.1}) corresponds to the work done by $p$ corner forces $[M_{t}(u)(P_{i})]$($i=1,
2, \ldots, p$), acting through $p$ corner displacements $u(P_{i})$
\citeup{chen1,chen2}.
For a sufficiently smooth function $u(t)$, the second and third boundary conditions in  \eqref{eq1} are satisfied on $\Gamma_1/P$,
and
\begin{align}\label{eq4.3}
	[M_{t}(u)(t,P_{i})]   =
	k_i u_t(t,P_{i}) \;\; \; \mbox{for} \;\; P_{i}\in \partial\Omega/ \overline{\Gamma}_{0},
\end{align}
where $k_i \geq 0  \; (i=1,\cdots,p) $ are  feedback control parameters associated with the corners.
A direct  computation then shows that for sufficiently smooth functions $u$ and $u^{\prime}$, the energy of the system on the
domain with corners --- particularly one that does not satisfy (G)--- still defined by (\ref{energy}), is
nonincreasing:
\begin{align}
	{\frac{d}{dt}}E(t)  =-d_1\|u_{t}\|_{\Gamma_{1}}^{2}
	-d_2\|\partial_{\nu}u_{t}\|_{\Gamma_{1}}^{2}- \sum_{P_i \in \partial\Omega/ \overline{\Gamma}_{0}}k_i |u_t(P_i)|^2 \leq 0.
	\label{eq4.4}
\end{align}

%{\em Remark 4.1 }
\begin{remark}
	\label{rem4.1} \rm
	More precisely, for sufficiently smooth
	functions $u$ and $u^{\prime}$, we can deduce the following
	boundary conditions at corner points~\cite{chen2}:
	\begin{align*}
		[M_{t}(u)(t,P_{i})] =
		\begin{cases}
			u_t(t,P_{i}) & \text{for } P_{i} \in \partial\Omega \setminus \overline{\Gamma}_{0},\\[4pt]
			0 & \text{for } P_{i} \in \overline{\Gamma}_{0}.
		\end{cases}
	\end{align*}
\end{remark}

%%{\em Remark 4.2 }
%\begin{remark}
%\label{rem4.2} \rm
%Boundary conditions (\ref{eq4.2}) and (\ref{eq4.3}) require
%that $u$ and $u^{\prime}$ are sufficiently smooth. In fact, in
%order for the pointwise limits of the twisting moments
%$M_{t}(u)(P_{i}^{+})$, $M_{t}(u)(P_{i}^{-})$ and the pointwise
%values $u^{\prime}(P_{i})$ to exist ($i=1,2,\ldots,l$), the
%sufficient conditions are
%$$
%\left\{
%\begin{array}{ll}
%u\in C^{2, \alpha_{1}}(\Gamma), & 0<\alpha_{1}<1,\\
%\noalign{\medskip} u^{\prime}\in C^{0, \alpha_{2}}(\Gamma), &
%0<\alpha_{2}<1.
%\end{array}
%\right.
%$$
%If $\displaystyle u\in V^{{7\over 2}+\varepsilon}$, $
%0<\varepsilon \leq {1\over 2}$, and $u^{\prime} \in V^{2}$, there
%would be no problem in (\ref{eq4.2}) and (\ref{eq4.3}) from the
%imbedding theorem on the domain with Lipschitz boundary
%\cite{grisvard}. However, since the boundary $\Gamma$ contains
%corners, the classical regularity results for elliptic boundary
%problems may no longer be valid \cite{bey,grisvard}.
%
%
\medskip

To develop a unified abstract framework for system \eqref{eq1} that encompasses both cases above and circumvents the technical difficulties arising from the possible loss of solution regularity, we now introduce several key function spaces and operators. First, note that the sesquilinear form $a(\cdot)$ is an equivalent norm on  $V $ since
$\Gamma_0 \not= \emptyset$  \citeup{JE,Rao2}. We define two symmetric, positive-definite, continuous sesquilinear forms as follows.
\begin{align*}
	\begin{split}
		&b_1(u,v) = d_1 \int_{\Gamma_1}\partial_{\nu}u  {\partial_{\nu}v}\; d\Gamma
		+d_2\int_{\Gamma_1} u  {v}\; d\Gamma ,\;\;\forall\; u ,v \in V,\\
		&b_2(u,v) = d_1 \int_{\Gamma_1}\partial_{\nu}u  {\partial_{\nu}v}\; d\Gamma
		+d_2\int_{\Gamma_1} u  {v}\; d\Gamma
		+ \sum_{i=1}^{p} k_i u(P_i) {v(P_i)},\;\;\forall\; u ,v \in V,\\
		&c(u,v) = \int_{\Omega}u  {v} \;dx + \rho \int_{\Gamma_1}u {v} \;d\Gamma
		+ J\int_{\Gamma_1}\partial_{\nu}u {\partial_{\nu}v}\;d\Gamma,\;\;\forall\; u ,v\in V.
	\end{split}
\end{align*}
Then the closed loop system \eqref{eq1} (or by the system \eqref{eq1} and \eqref{eq4.3})  can be rewritten as the following variational evolution equation:
\begin{align}
	c(u_{tt}, \xi) + a(u,\xi) + b_j(u_t,\xi)  = 0,\quad \forall\; \xi \in V.
\end{align}
Here, the index $j=1$ corresponds to the plate system on a domain satisfying the assumption stated in Lemma \ref{Le2.2}, while  $j=2$ corresponds to the system on a domain that fulfills the assumption in Lemma \ref{lemma-green}.
We denote the completion of the space $V$ normed by $c^{\frac{1}{2}}(\cdot,\cdot)$ as $H$ and let $V^*$ be the dual of $V$ pivotal to $H$. It is clear that $V\hookrightarrow H= H^*\hookrightarrow V^*$, where $\hookrightarrow$ denotes the continuous dense injection.

We define several operators induced by the sesquilinear forms $a(\cdot, \cdot)$, $b_j(\cdot, \cdot)$ and $c(\cdot, \cdot)$ as follows,
	\begin{align*}
		\begin{split}
		&A :V\rightarrow V^*,\; \langle A u , v \rangle_{V^*\times V} = a(u, v),\;\qquad\quad\forall\; u, v\in V,\\
		&B_j :V\rightarrow V^*,\; \langle B_j u , v \rangle_{V^*\times V} = b_j(u, v),\qquad\forall\; u, v\in V,\\
		&\mathcal C  : H \rightarrow H^*,\; \langle \mathcal  C u , v \rangle_{H^*\times H} = c(u, v),\;\quad\qquad\forall\; u, v\in H.
		\end{split}
	\end{align*}
Then, \eqref{eq1} (or \eqref{eq1} and \eqref{eq4.3}) can be rewritten respectively as
\begin{align}
	\mathcal C u_{tt} + B_j u_t + A u = 0, \quad \mbox{in}\; V^*, \;\; j=1,2.
\end{align}
Now we introduce the energy Hilbert space
$$\mathcal{H}  = V\times H$$
with the norm
$$\|(u,v)\|_{\mathcal{H} } = ( a(u) + \|v\|_{\Omega}^2 + \rho\|v\|_{\Gamma_1}^2 + J\|\partial_{\nu}v\|_{\Gamma_1}^2 )^{\frac{1}{2}},\quad \forall\;(u,v)\in\mathcal H,$$
and $\mathcal{H} ^* = V^*\times H^*= V^*\times H$ be the dual of $\mathcal{H} $. Define two operators
\begin{align*}
	\begin{split}
		&\hat C =\Bigg(\begin{matrix}A&0\\0&\mathcal C\end{matrix}\Bigg)
		\mbox{ is the isomorphism of $\mathcal H$ onto $\mathcal H^*$},\\
		&\hat A_j =\Bigg(\begin{matrix}0&A\\-A&-B_j \end{matrix}\Bigg): D(\hat A_j)\rightarrow\mathcal H^*
	\end{split}
\end{align*}
with domain
$$D(\hat A_j )= \{(u,v)\in\mathcal H\;|\; v\in V,\;Au+B_j v\in H\},\;\; j=1,2.$$
Now we can define
\begin{align}\label{def}
	\mathcal A_j  = \hat C^{-1}\hat A: \mathcal{D}(\mathcal A_j)  = \mathcal{D}(\hat A_j )\rightarrow \mathcal H,\;\; j=1,2.
\end{align}
Assume $z(t) \doteq (u(t),u'(t))$ and $z_0 = (u_0,u_1)$. Then system \eqref{eq1} (or \eqref{eq1} and \eqref{eq4.3})  can be rewritten respectively as
\begin{align}
	\begin{cases}\displaystyle
		\frac{d}{dt}z(t) =\mathcal{A}_j  z(t),\;\; j=1,2,\\
		z(0)=z_{0}\in \mathcal{H} .
	\end{cases}
\end{align}

The well-posedness of system \eqref{eq1} (or \eqref{eq1} and \eqref{eq4.3})  is given in the following lemma. A proof is provided here for the sake of completeness, though the argument is standard.

\begin{lemma}\label{well}
	The unbounded linear operator $\mathcal{A}_j$ generates a $C_0$-semigroup $e^{\mathcal{A}_j t}$ of contractions on $\mathcal{H} $, and $0\in\rho(\mathcal{A}_j), \; j=1,2$.
\end{lemma}
\begin{proof} We only prove the case $j=2$.
	For $z=(u,v)\in \mathcal{D}(\mathcal{A}_2)$, we have that
	\begin{align}\label{haosan}
		\begin{split}
			\Re(\mathcal{A}_2 z,\; z)_{\mathcal{H} } &= (u,v)_{V} + c (-\mathcal C^{-1}(A u + B_2 v), v)_{H} \\
			& = - d_1\int_{\Gamma_1}|\partial_{\nu} v|^2 d\Gamma - d_2 \int_{\Gamma_1}|v|^2 d\Gamma - \sum_{i=1}^p k_i  |v(P_i)|^2.
		\end{split}
	\end{align}
	Thus $\mathcal{A}_2$ is dissipative.
	
	Furthermore, for any $(f,g)\in \mathcal H$, consider  $\mathcal A_2(u,v)=(f,g)$, i.e.,
	\begin{align}\label{0}
		\begin{cases}
			v = f \in V,\\
			\mathcal C ^{-1}A u + \mathcal C^{-1}B_2 v = -g\in H.
		\end{cases}
	\end{align}
	It is clear that $v\in V$ and $A u + B_2 v = \mathcal C g\in H$.
	We deduce from \eqref{0}that
	\begin{align}\label{lax}
		a(u,\phi) = - b_2(f,\phi) - c(g,\phi),\qquad\forall\;\phi\in V.
	\end{align}
	Due to the Lax-Milgram theorem, \eqref{lax} admits a unique solution $u\in V$ for any
	given $(f,g)\in\mathcal H$.
	Hence, $0\in\rho(\mathcal{A}_1)$.
\end{proof}

\medskip
The main result of this paper is as follows.
\begin{theorem}\label{Th1}
	Let $\Omega\subset\mathbb{R}^2$ satisfy conditions in Lemma \ref{Le2.2} or  Lemma \ref{lemma-green}.   Suppose   there exists a point $x_0 \in \mathbb{R}^{2}$  such that, when setting $\mathbf{m} (x)=x-x_0$, it holds
	\begin{align}\label{assume-c}
		\mathbf{m} \cdot \nu \big|_{\Gamma_1} \geq \gamma >0,  \quad
		\mathbf{m} \cdot \nu\big|_{\Gamma_0} \leq 0.
		\tag{\bf H}\end{align}
	Then, there exists a positive constant $C$ such that for any $z_0\in\mathcal{D}(\mathcal{A}_j)$,
	$$
	\|e^{\mathcal{A}_j t}z_0\|_{\mathcal{H} }\leq Ct^{-{\frac{1}{2}}}
	\|z_0\|_{\mathcal{D}(\mathcal{A}_j)},\;\; t\geq 1 ,\;\;j=1,2.
	$$
\end{theorem}

\vskip10pt
\begin{figure}[htbp]
	\centering
	\subfigure[]{
		\begin{minipage}[t]{0.245\textwidth}
			\centering
			\includegraphics[height=3cm,width=3.5cm]{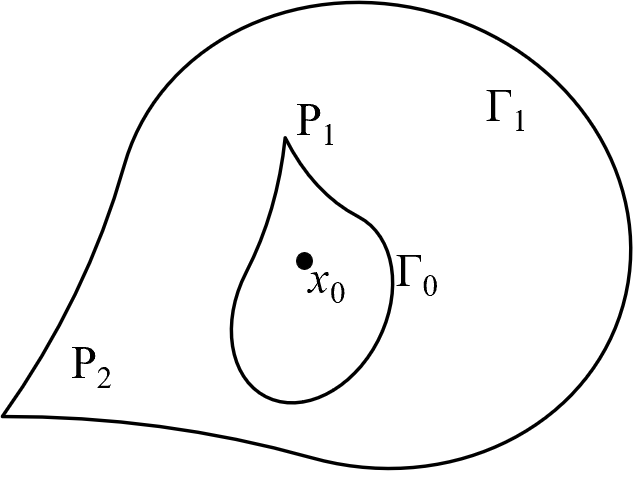}
			%\caption{fig2}
		\end{minipage}%
	}%
    \centering
    \caption{A domain with corners satisfying (i) $\Gamma_1$ and $\Gamma_0$ are disjoint;  (ii) the  hypotheses (G) and (H) hold.}
    	\label{fig111}
\end{figure}
\vskip10pt 
\begin{figure}[htbp]
	\centering
	\subfigure[]{
		\begin{minipage}[t]{0.245\textwidth}
			\centering
			\includegraphics[height=3cm,width=3.5cm]{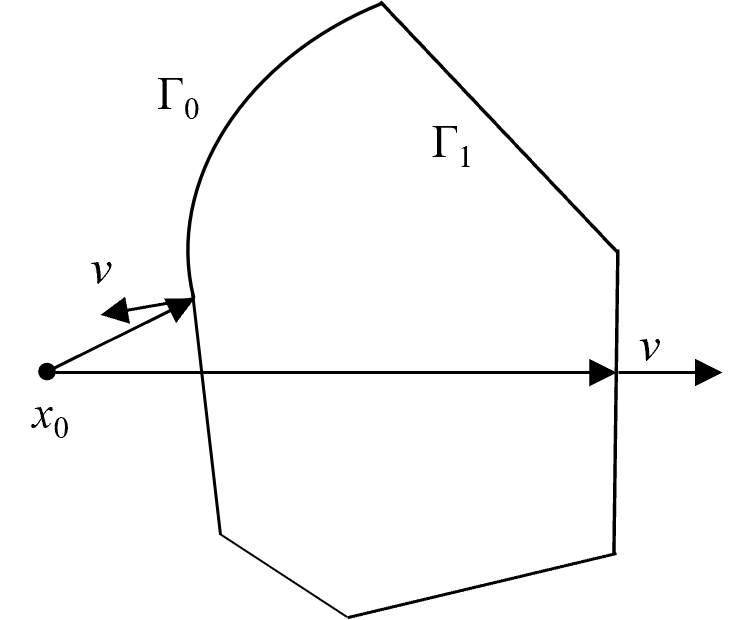}
			%\caption{fig1}
		\end{minipage}%
	}%
		\caption{A domain with corners satisfying (i) $\Gamma_1$ and $\Gamma_0$ shares two end points, both of which are corner points;  (ii) the geometric hypothesis (H) holds.}
	\label{fig22}
\end{figure}
%\begin{figure}[htbp]
%\centering
%\subfigure[]{
%	\begin{minipage}[t]{0.245\textwidth}
%	\centering
%	\includegraphics[height=3cm,width=3.5cm]{fig111.png}
%	%\caption{fig1}
%	\end{minipage}%
%	}%
%\centering
%\caption{A domain with corners satisfying (i) $\Gamma_1$ and $\Gamma_0$ are disjoint;  (ii) the  hypotheses (G) and (H) hold.}
%\label{fig2}
%\end{figure}
%
%\vskip10pt
%
%\begin{figure}[htbp]
%	\centering
%	\subfigure[]{
%		\begin{minipage}[t]{0.245\textwidth}
%			\centering
%			\includegraphics[height=3cm,width=3.5cm]{fig22.png}
%			\caption{fig1}
%		\end{minipage}%
%	}%
%	\centering
%	\caption{A domain with corners satisfying (i) $\Gamma_1$ and $\Gamma_0$ shares two end points, both of which are corner points;  (ii) the geometric hypothesis (H) holds.}
%	\label{fig3}
%\end{figure}

\begin{remark}
	In Ref. \cite{Rao1,Rao2}, the stabilization of both one- and two-dimensional versions of system \eqref{eq1} was analyzed. Specifically, Rao established polynomial stability with a decay rate of $t^{1/2} $ for system \eqref{eq1} on a bounded domain $\Omega \subset {\mathbb R}^2$ having a smooth boundary and satisfying the geometric condition (H).
	Our work extends this stability result to problems defined on non-smooth domains. More precisely, we show that:  (i) when the interior angles at the corner points are sufficiently small (see, e.g., Fig. \ref{fig111}), polynomial stability is likewise achieved; (ii) when the assumptions of Lemma \ref{lemma-green} hold (see, e.g., Fig. \ref{fig22}), we apply controls on the corner points, and proved that the semigroup corresponding to the closed loop system decays polynomially at the rate $t^{-{1/2}} $. In subsequent  work, we shall show the explicit decay rate for  system \eqref{eq1} and \eqref{eq4.3} on particular domains, such as rectangles \citeup{sz}.
\end{remark}

\medskip

The following lemmas are useful.

\begin{lemma}\label{cor} (See Ref. \cite{chen1,chen2})
	For sufficiently smooth functions $u$ and $v$, it holds
	\begin{align}
		\begin{split}
			\int_{\Omega}\Delta^{2}u {(\mathbf{m}\cdot\nabla u)}dx=&
			\;a(u)+\frac{1}{2}\int_{\partial\Omega}
			(\mathbf{m}\cdot\nu)q(u)d\Gamma
			+\sum_{i=1}^p[M_t(u)(P_i)] ({\mathbf{m}\cdot\nabla u(P_i)})\\
			&+\int_{\partial\Omega}\mathcal{B}_{2}u {(\mathbf{m}\cdot\nabla u)}-\mathcal{B}_{1}u {\partial_{\nu}(\mathbf{m}\cdot\nabla u)}d\Gamma,
		\end{split}
	\end{align}
	where
	\begin{align}\label{bu}
		q(u)=\bigg|\frac{\partial^{2}u}{\partial x_{1}^{2}}\bigg|^{2}+
		\bigg|\frac{\partial^{2}u}{\partial x_{2}^{2}}\bigg|^{2}
		+2\mu\frac{\partial^{2}u}{\partial x_{1}^{2}}\frac{\partial^{2}u}{\partial x_{2}^{2}}
		+2(1-\mu)\bigg|\frac{\partial^{2}u}{\partial x_{1}\partial x_{2}}\bigg|^{2}.
	\end{align}
\end{lemma}

\begin{lemma}\label{lemma:methods} (See Ref. \cite{BorTom,LiuRao05,RSS})
	Let $e^{t\mathcal{A}}$ be a $C_0$-semigroup of contractions on a Hilbert space $H$ with
	generator $\mathcal{A}$ such that $i\mathbb{R}\subseteq \rho(\mathcal{A})$. Then
	the semigroup is polynomially stable with the order ${\beta}$, i.e.,  there exists a constant $C>0$ such that
	\begin{align}
		\|e^{t\mathcal{A}}z_0\|_{H} \leq \frac{C}{t^{\beta}} \|z_0\|_{D(\mathcal{A})},
		\quad   z_0 \in D(\mathcal{A}), \;  t>1.
	\end{align}
	if and only if
	\begin{align}
		\label{huang}
		\limsup\limits_{|\lambda|\rightarrow\infty}\|\lambda^{-\frac{1}\beta}(i\lambda-\mathcal{A})^{-1}\| < \infty.
	\end{align}
\end{lemma}

\section{Proof of main result}

The proof of Theorem \ref{Th1} is presented in this section. For clarity, we only provide the detailed argument only for the case  $j=2$. The case $j=1$ follows analogously.
According to Lemma \ref{lemma:methods}, to prove \eqref{huang}, it suffices to show the existence of a constant $r > 0$ such that, for  $\theta\ge2 , $
\begin{align}\label{equalcon}
	\inf_{\|z\|_{\cal H}=1,\omega\in\mathbb{R}}\omega^{\theta}\|i\omega z-\mathcal{A}_2 z\|_{\mathcal{H}}\geq r.
\end{align}
If \eqref{equalcon} fails, there exists $\{\omega_{n},z_{n}\}_{n=1}^{\infty}\subset\mathbb{R}\times \mathcal{D}(\mathcal{A}_2)$ with
\begin{align}\label{unitnorm}
	\|z_{n}\|_{\mathcal{H}}=\|(u_{n},\;v_{n},\;\xi_{n},\;\eta_{n})\|_{\mathcal{H}}=1,
\end{align}
such that
\begin{align}\label{zero}
	f_{n } = \big(f_{1,n},f_{2,n},g_{1,n},g_{2,n}\big) \doteq \omega_{n}^{\theta}(i\omega_{n}I-\mathcal{A}_2)z_{n} = o(1) \;  \mbox{in } \; \mathcal{H},
\end{align}
which is equivalent to as $n\to \infty,$
\begin{subequations}
	\begin{align}
		&f_{1,n}=\omega_{n}^{\theta}(i\omega_{n}u_{n}-v_{n})=o(1), &&  \mbox{in} \;\; V,\label{3.5a}\\
		&f_{2,n}=\omega_{n}^{\theta}(i\omega_{n}v_{n}+\Delta^{2}u_{n})=o(1), & & \mbox{in} \;\; L^{2}(\Omega),\label{3.5b}\\
		& g_{1,n}=\omega_{n}^{\theta}[(i\omega_{n}J+d_1)\xi_{n}+\mathcal{B}_{1}u_{n}]=o(1), & & \mbox{in} \;\; L^{2}(\Gamma_{1}),\label{3.5c}\\
		& g_{2,n}=\omega_{n}^{\theta}[(i\omega_{n}\rho+d_2)\eta_{n}-\mathcal{B}_{2}u_{n}]=o(1),
		& &  \mbox{in} \;\; L^{2}(\Gamma_{1})\label{3.5d},
	\end{align}
\end{subequations}
and
\begin{align}\label{35e}
	M_t(u_n)(P_i) = k_i v_n (P_i) .
\end{align}
Furthermore, by the dissipativeness of operator $\mathcal{A}_2$ and \eqref{zero}, one has
\begin{align} \label{diss}
	Re\big((i\omega_{n}-\mathcal{A}_2)z_{n},\; z_{n}\big)_{\mathcal H}
	=d_1\|\xi_{n}\|_{L^2(\Gamma_{1})}^{2}
	+d_2\|\eta_{n}\|_{L^2(\Gamma_{1})}^{2}
	+\sum_{i=1}^p k_i |v_n(P_i)|^2= \omega_{n}^{-\theta}o(1).
\end{align}
Consequently, by \eqref{3.5a} and \eqref{diss},
\begin{align}\label{diss1}
	\|u_{n}\|_{L^2(\Gamma_{1})},\;\; \|\partial_\nu u_{n}\|_{L^2(\Gamma_{1})}=\omega_{n}^{-1-\frac{\theta}{2}}o(1).
\end{align}
%Our goal is to reach a contradiction by showing $\|X_{n}\|_{H}=o(1)$. What follows in this section, we shall prove
In the following, we shall show that
\begin{align}
	a(u_{n}),\; \; \|v_{n}\|_{L^2(\Omega)}    =o(1).
\end{align}
Then, combining this with \eqref{diss} yields  $\|z_{n}\|_{\mathcal{H}}=o(1)$,
which  leads to a contradiction.

Multiplying \eqref{3.5a} and \eqref{3.5b} by $\mathbf{m}\cdot\nabla u_n$ separately and summing the results, we obtain
\begin{align} \label{mu1}
	(\Delta^{2}u_{n},\mathbf{m}\cdot\nabla u_{n})_{L^2(\Omega)}-(v_{n},\mathbf{m}\cdot\nabla v_{n})_{L^2(\Omega)} =\omega_{n}^{-\theta}o(1).
\end{align}
By Lemma \ref{cor} and the boundary conditions, it follows that
\begin{align}\label{mu22}
	\begin{split}
		\Re&\bigg[a(u_{n}) + \|v_n\|^2
		+\frac{1}{2}\int_{\partial\Omega}(\mathbf{m}\cdot\nu)q(u_{n})d\Gamma\\
		&
		+\int_{\partial\Omega}\mathcal{B}_{2}u_{n} {(\mathbf{m}\cdot\nabla u_{n})}
		-\mathcal{B}_{1}u_{n} {\partial_{\nu}(\mathbf{m}\cdot\nabla u_{n})}d\Gamma\\
		&
		-\frac{1}{2}\int_{\Gamma_1}(\mathbf{m}\cdot\nu)|v_{n}|^{2}d\Gamma+\sum_{i=1}^{p}k_i v_n(P_i)
		({\mathbf{m}_1\cdot\nabla u_n(P_i)})
		\bigg]=\omega_{n}^{-\theta}o(1).
	\end{split}
\end{align}
Since $u_n=\partial_{\nu}u_n=0$ on $\Gamma_0$, it follows that
\begin{align}
	\nabla u_n =0 ,\;\;\mathcal B_1 u_n = \Delta u_n,\;\;\partial_{\nu}(\mathbf m\cdot\nabla u_n) = (\mathbf m\cdot \nu)\Delta u_n,\;\;q(u_n)=(\Delta u_n)^2\;\;\text{on}\;\Gamma_0.
\end{align}
Then, it follows from \eqref{assume-c}, \eqref{diss} and  \eqref{mu22} that
\begin{align}\label{mu2}
	\begin{split}
		a(u_{n}) + \|v_n\|^2 \le &
		-\frac{1}{2}\int_{\Gamma_1}(\mathbf{m}\cdot\nu)q(u_{n})d\Gamma
		- \Re \int_{\Gamma_1} \big[\mathcal{B}_{2}u_{n} {(\mathbf{m}\cdot\nabla u_{n})} \\
		&
		- \mathcal{B}_{1}u_{n} {\partial_{\nu}(\mathbf{m}\cdot\nabla u_{n})}\big]d\Gamma
		- \sum_{i=1}^{p}k_i v_n(P_i)
		({\mathbf{m}_1\cdot\nabla u_n(P_i)})
		+\omega_{n}^{-\theta}o(1).
	\end{split}
\end{align}
From  \eqref{3.5d} and \eqref{diss1}, we deduce
\begin{align}
	\label{b1}
\begin{split}
		\Big|\int_{\Gamma_{1}}\mathcal{B}_{2}u_{n} {(\mathbf{m}\cdot\nabla u_{n})}d\Gamma
	\Big|
	&\leq C\|(i\omega_n \rho +d_2)\eta_{n}\|_{L^2(\Gamma_1)}(\|u_n\|_{L^2(\Gamma_1)}+\| \partial_\nu u_n\|_{L^2(\Gamma_1)})\\
	&  = \omega_n^{-\theta}o(1).
\end{split}
\end{align}
Furthermore, note that
$$  \partial_{\nu }(\mathbf m\cdot\nabla u_n)
= \partial_{\nu } u_n  +
m_1\nu_1{\frac{\partial^{2}u_n}{\partial x^{2}_{1}}}
+m_2\nu_2{\frac{\partial^{2}u_n}{\partial x^{2}_{2}}}   +(m_1\nu_2+m_2\nu_1)
{\frac{\partial^{2}u_n}{\partial x_{1}\partial
		x_{2}}}.
$$
Consequently,
\begin{align}
	\label{b2}
	\begin{split}
		\Re  \int_{\Gamma_{1}}\mathcal{B}_{1}u_{n} {\partial_{\nu}(\mathbf{m}\cdot\nabla u_{n})}d\Gamma
		=& \displaystyle
		\int_{\Gamma_{1}}\Big[(i\omega_n J+d_1)\xi_n +\omega_n^{-\theta} g_{1,n}\Big] \Big[\partial_{\nu}u_{n}+   m_1\nu_1{\frac{\partial^{2}u_n}{\partial x^{2}_{1}}}
		\\ & \displaystyle +m_2\nu_2{\frac{\partial^{2}u_n}{\partial x^{2}_{2}}}
		+
		(m_1\nu_2+m_2\nu_1)
		{\frac{\partial^{2}u_n}{\partial x_{1}\partial
				x_{2}}}\Big]d\Gamma.
	\end{split}
\end{align}
Using \eqref{diss}-\eqref{diss1} in the above inequality yields
\begin{align}\label{b3}
	\begin{split}
		&\Bigg|\int_{\Gamma_{1}}\mathcal{B}_{1}u_{n} {\partial_{\nu}(\mathbf{m}\cdot\nabla u_{n})}d\Gamma\bigg|\\
		\leq  & \displaystyle
		\frac{\gamma(1-\mu)}{4} \int_{\Gamma_{1}}  \bigg(\bigg|{\frac{\partial^{2}u_n}{\partial
				x^{2}_{1}}}\bigg|^2 +\bigg|{\frac{\partial^{2}u_n}{\partial
				x^{2}_{2}}}\bigg|^2 +2 \bigg|{\frac{\partial^{2}u_n}{\partial
				x_{1}\partial x_{2}}}\bigg|^2\Bigg)d\Gamma
		+ \omega_n^{2-\theta} o(1).
	\end{split}
\end{align}
where the constant ${{\gamma(1-\mu)\over4}}$, which comes from the Cauchy-Schwarz inequality, will be used in \eqref{mu4}.
Note that by assumption \eqref{assume-c}, we have on
$\Gamma_1$
\begin{align}
	\label{be4}
	(\mathbf{m}\cdot \nu)q(u_n)\ge \gamma (1-\mu) \bigg(\bigg|{\frac{\partial^{2}u_n}{\partial
			x^{2}_{1}}}\bigg|^2 +\bigg|{\frac{\partial^{2}u_n}{\partial
			x^{2}_{2}}}\bigg|^2 +2 \bigg|{\frac{\partial^{2}u_n}{\partial
			x_{1}\partial x_{2}}}\bigg|^2\bigg).
\end{align}

\noindent
Moreover, according to \cite{chen1,chen2},
\begin{align}\label{point-m}
	\begin{split}
		\Bigg|\sum_{i=1}^{p} k_iv_n(P_i)
		({\mathbf m \cdot\nabla u_n(P_i)})\Bigg|
		& \leq
		\Bigg|\sum_{i=1}^p
		\bigg( \frac{k_i^2}{4\varepsilon}|v_n(P_i)|^2 + \varepsilon |\mathbf m \cdot \nabla  u_n (P_i)|^2\bigg)\\
		& \leq   \omega_n^{-\theta }o(1)+ \varepsilon \Big[a(u_n) + \int_{\Gamma_1}(\mathbf{m}_1\cdot\nu)
		q(u_n)d\Gamma \Big].
	\end{split}
\end{align}

\noindent
In summary,  substituting \eqref{diss}, \eqref{b1}, \eqref{b3}, \eqref{be4}, \eqref{point-m} into \eqref{mu2} and taking $\varepsilon$ small enough, we obtain for $\theta\geq 2, $
\begin{align}\label{mu4}
	a(u_{n}) + \|v_n\|^2_{L^2(\Omega)} \leq   o(1).
\end{align}

Finally, we show that $i \mathbb R  \subset \rho(\mathcal A_2)$.
From Lemma \ref{well}, we have $0\in\rho({\mathcal{A}_2})$. Thus, we can define
$$\widetilde{\omega}:=\sup \{ R>0: [-iR, iR]\subset \rho(\mathcal{A}_2)\}.$$
To complete the proof, it suffices to show $\widetilde{\omega}=\infty$.
Suppose, for contradiction, that $0<\widetilde{\omega}<\infty$.
Then there exist a sequence
$\omega_n \to \widetilde{\omega}$  and a sequence
$z_n \in D(\mathcal{A}_2)$ with $\|z_n\|_\mathcal{H}=1$ such that
\begin{align}\label{+331-}
	\Vert (i\omega_n I-\mathcal{A}_2) z_n\Vert_\mathcal{H} \to 0,\quad n\to\infty,
\end{align}
Applying the same reasoning as in the proof of \eqref{huang} and using $\omega_n \to \widetilde{\omega}$ leads to a contradiction, which completes the proof.

\section*{Acknowledgments}
The authors are grateful to Professor Otared Kavian for valuable discussions and helpful comments.

The project is supported by the National Natural Science Foundation of China (grants No. 12271035, 12131008) and
Beijing Municipal Natural Science Foundation (grant No. 1232018).

%\section*{Conflict of Interest}
%
%The authors declare no conflict of interest.

% XX is an editor-in-chief for Acta Mathematicae Applicatate Sinica (English Series) and was not involved in the editorial review or the decision to publish this article. All authors declare that there are no competing interests.

Y.N. Sun, School of Mathematics and Statistics, Beijing Institute of Technology, Beijing, 100081, P.R. China

Email address: yanansun@bit.edu.cn

Q. Zhang, 
School of Mathematics and Statistics, Beijing Institute of Technology, Beijing, 100081, P.R. China

Email address: zhangqiong@bit.edu.cn

\end{document}